\documentclass[11pt]{amsart}
\usepackage[UKenglish]{babel}
\usepackage[utf8]{inputenc} 
\usepackage{amsthm} 
\usepackage{amsfonts} 
\usepackage{mathtools}
\usepackage{enumitem} \setlist[enumerate]{label={\upshape(\arabic*)}}
\usepackage{amssymb}
\usepackage[numbers]{natbib}
\usepackage{url}
\usepackage{mathrsfs}
\usepackage{geometry}
\usepackage{bbm}
\usepackage{tikz-cd}
\usepackage{tikz}
\usepackage{color}
\usepackage{framed}
\usepackage{hyperref}
\hypersetup{
    colorlinks=true, 
    linkcolor=blue, 
    urlcolor=red, 
    citecolor=[rgb]{0,0.7,0},
    linktoc=all 
}

\theoremstyle{definition}
\newtheorem{defn}{Definition}[section]
\newtheorem{prop}[defn]{Proposition}

\newtheorem{thm}[defn]{Theorem}

\newtheorem{conj}[defn]{Conjecture}

\newtheorem{question}[defn]{Question}

\newcommand{\R}{\mathbb{R}}

\newcommand{\Z}{\mathbb{Z}}

\DeclareMathOperator{\conv}{conv}

\def\1{\mathbf{1}}

\newcommand{\ZZ}{{\mathbb{Z}}}
\newcommand{\N}{{\mathbb{N}}}

\newcommand{\hs}{h^*}

\title[A very ample lattice polytope with a non-unimodal $h^*$-vector]{A very ample lattice polytope with a non-unimodal $h^*$-vector}
\author[J.~Hofscheier]{Johannes Hofscheier}
\author[V.~Kurylenko]{Vadym Kurylenko}
\author[B.~Nill]{Benjamin Nill}

\address[J.~Hofscheier]{School of Mathematical Sciences\\University of Nottingham\\ Nottingham\\NG7 2RD\\UK}
\email{johannes.hofscheier@nottingham.ac.uk}

\address[V.~Kurylenko, B.~Nill]{Faculty of Mathematics, Otto-von-Guericke-Universit\"at Magdeburg, Universit\"atsplatz 2, 39106 Magdeburg, Germany.}
\email{\{vadym.kurylenko, benjamin.nill\}@ovgu.de}


\begin{document}

\setlength{\parindent}{0pt}
\begin{abstract}
    Lattice polytopes are called \emph{very ample} if for every sufficiently large $k$ every lattice point of height $k$ in the cone over the lattice polytope is the sum of $k$ lattice points of height $1$.
    This is a weakening of the well-known integer decomposition property (also called IDP).
    We give an example of a very ample lattice polytope whose $h^*$-vector is non-unimodal.
    Here, the $h^*$-vector is the coefficient vector of the numerator of the Ehrhart series of the lattice polytope.
    This answers a question of  Ferroni and Higashitani, as well as a related question by Balletti.
    The main question whether IDP lattice polytopes have unimodal $h^*$-vector is still open.
    The example was found using ChatGPT 5.6 Sol.
    It is just the Cartesian square of a lattice polytope belonging to a class of very ample examples constructed by Laso\'{n} and Micha\l ek. 
\end{abstract}

\maketitle

\thispagestyle{empty}

\section{Introduction and main result}

Let us recall the main definitions\footnote{Much of this introduction is directly taken from~\cite{IDP-paper}.
We plan to merge both papers into one publication.}.
We refer to~\cite{beckbook} for references.
A \emph{lattice polytope} $P \subset \R^n$ is the convex hull of finitely many lattice points, i.e., elements of $\Z^n$.
The \emph{Ehrhart polynomial} of $P$ is the unique polynomial $E_P(t)$ with $E_P(k) = |k P \cap \Z^n|$ for $k \in \N_{\ge 1}$.
The \emph{$\hs$-polynomial} $\hs_P(t)$ of $P$ is defined as the numerator of its \emph{Ehrhart generating series}: 
\[
    \sum_{k=0}^\infty E_P(k) t^k = \frac{\hs_P(t)}{(1-t)^{d+1}},
\]
where $d$ denotes the dimension of $P$.
For $\hs_P(t)=\sum_{i=0}^d \hs_i t^i$, we call $(\hs_0, \ldots, \hs_d)$ the {\em $\hs$-vector} of $P$. 

\smallskip

The following properties of lattice polytopes have been intensively studied over the previous years.
We refer to the excellent survey papers~\cite{Braun-Survey} and~\cite{Ferroni}.
An \emph{IDP polytope} (an abbreviation for lattice polytopes with the \emph{Integer Decomposition Property}) is a lattice polytope $P$ such that for every $k \in \N_{\ge 2}$ each lattice point in $kP$ is a sum of $k$ lattice points in $P$.
If there exists some positive integer $N$ such that the previous property holds for all $k \ge N$, then $P$ is called \emph{very ample}.
We say the $\hs$-vector of $P$ is \emph{unimodal} if $\hs_0 \le \hs_1 \le \cdots \le \hs_j\ge \hs_{j+1} \ge \cdots \ge \hs_d$ for some $j \in \{0, \ldots, d\}$.
It is called \emph{log-concave} if $\hs_{i-1} \hs_{i+1} \le (\hs_i)^2$ for every $i \in \{1, \ldots, s-1\}$ where $s \in \ZZ_{\ge0}$ denotes the degree of the polynomial $h_P^*(t)$. 

\smallskip

The following conjecture is currently considered as the main open question in Ehrhart theory.
It is explicitly stated in~\cite[Conjecture~1.1]{Ferroni} and~\cite[Question~1.1]{Van}.

\begin{conj}\label{unimodality}
    Every IDP polytope has a unimodal $\hs$-vector.
\end{conj}

In~\cite{IDP-paper} the authors used machine learning to find an example of an IDP polytope with non-log-concave $h^*$-vector showing that one cannot strengthen the implication in the conjecture.
In this paper we used ChatGPT 5.6 Sol to find an example that shows that one also cannot weaken the assumption to very ampleness.
This completely answers Question~3.9 by Ferroni and Higashitani in~\cite{Ferroni}.

\medskip

Let us now describe the example. It is the Cartesian square of a lattice polytope constructed by Laso\'{n} and Micha\l ek in~\cite{Mat} to disprove several conjectures on IDP lattice polytopes.

\smallskip

For this, let $e_1, \ldots, e_{30}$ be the standard basis vectors of $\R^{30}$.
We define for $i=1, \ldots, 30$
\[
    u_i \coloneqq e_i+e_{i+1}\in\Z^{30}
\]
with indices cyclic modulo $30$, and set
\[
    Q := \conv\!\left(
         (u_1, 332), (u_1, 333), (u_i, 0),(u_i, 1) : 2 \le i \le 30
     \right) \subset \R^{30} \times \R.
\]

In the notation of \cite{Mat}, this is precisely $\mathcal{P}_{15,332}$, a lattice segmental fibration over the edge polytope of the even cycle with $30$ vertices.
\begin{prop}\label{prop}
    The lattice polytope $Q \subset \R^{31}$ has the following properties:
    \begin{enumerate}
        \item its dimension is $29$ and it has $60$ vertices,
        \item it is very ample but not IDP,
        \item its $h^*$-polynomial is
            \[
                h_Q^*(t)=1+30(t+\cdots+t^{14})+346t^{15}.
            \]
    \end{enumerate}
\end{prop}

\begin{proof}
    Dimension, number of vertices and the $h^*$-polynomial can be easily computed with any suitable software package.
    Very ampleness of $Q$ has been shown in \cite{Mat} (more precisely, Theorem~3 and Proposition~6).
    It is not IDP by Theorem~12 in~\cite{Mat}.
\end{proof}

Now, as in Balletti's paper \cite{Balletti} taking a Cartesian product of $Q$ with itself does the trick.

\begin{thm} \label{thm:main}
    $Q \times Q$ is a $58$-dimensional very ample, non-IDP lattice polytope with $3600$ vertices and non-unimodal $h^*$-vector.
\end{thm}

\begin{proof}
    Clearly, Cartesian products preserve very ampleness as well as being non-IDP. Note that for all nonnegative integers $k$ we have $E_{Q \times Q}(k)=(E_Q(k))^2$. Therefore, for integers $r$ we get for the $r$th-coefficient of $h^*_{Q \times Q}(t)$
    \[h_r^*=\sum_{m=0}^r (-1)^{r-m} \binom{59}{r-m} E_Q(m)^2,\]
    where by Proposition~\ref{prop}
    \[E_Q(m)=\binom{m+29}{29}+\left(\sum_{i=1}^{14} 30 \binom{m+29-i}{29}\right) + 346 \binom{m+29-15}{29}.\] 
    Therefore, we can verify non-unimodality by evaluating
    
    \begin{align*}
        \hs_{25} &= 1631347108606059869973,\\
        \hs_{26} &= 1627940480253228812568,\\
        \hs_{27} &= 1628167830322084291128. \qedhere
    \end{align*}
\end{proof}

$Q = \mathcal{P}_{15,332}$ produces the minimal example given by the above construction in terms of dimension and volume. Unimodality also fails when we consider $Q=\mathcal{P}_{15,a}$ for $a\in\{332, \ldots, 338\}$. 

As this is also an example of two very ample lattice polytopes such that their Cartesian product does not have a unimodal $h^*$-vector, this also answers Question~5.1(b) by~\cite{Balletti} and strengthens his main result (Theorem~4.4).
Note the similarity of the $h^*$-vectors of $Q$ with the ones found by Balletti using genetic algorithms.
\begin{question}
    Is it possible to find a $d$-dimensional very ample polytope whose $h^*$-vector is not-unimodal, and there is an index $i>d/2$ where the unimodality is violated? 
\end{question}

\begin{figure}[ht]
    \centering
    \includegraphics[width=1.05\textwidth]{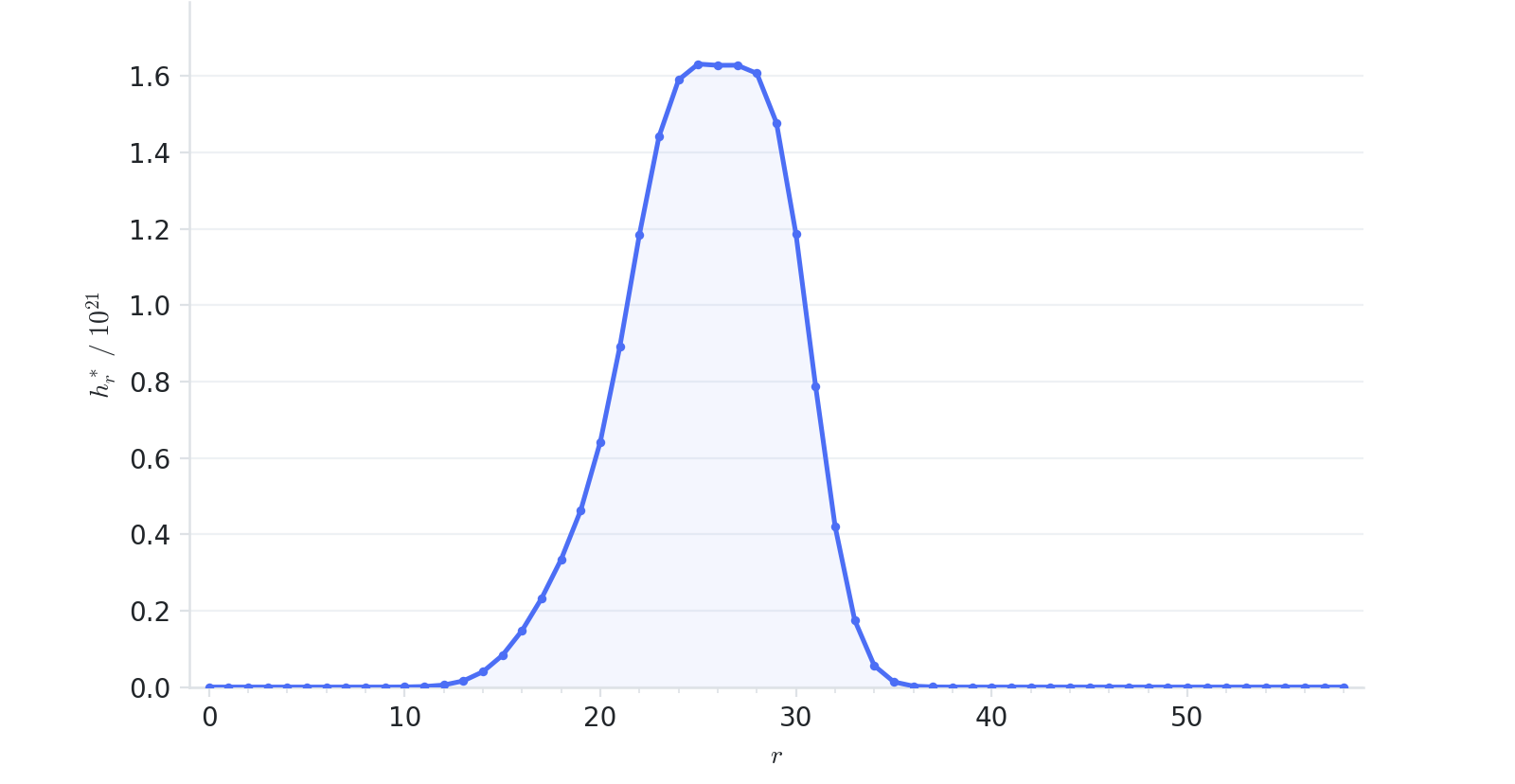}
    \caption{$h^*$-vector from Theorem \ref{thm:main} }
    \label{fig:example}
\end{figure}

\subsection*{Acknowledgment}
The example was found using ChatGPT 5.6 Sol.
This work is funded by the Deutsche Forschungsgemeinschaft (DFG, German Research Foundation) – 539867500 as part of the research priority program Combinatorial Synergies. 

\bibliographystyle{plain}
\bibliography{ref}

@Article{unimodal,
 Author = {Schepers, Jan and Van Langenhoven, Leen},
 Title = {Unimodality questions for integrally closed lattice polytopes},
 FJournal = {Annals of Combinatorics},
 Journal = {Ann. Comb.},
 ISSN = {0218-0006},
 Volume = {17},
 Number = {3},
 Pages = {571--589},
 Year = {2013},
doi={10.1007/s00026-013-0185-6}
 }

@incollection {Braun-Survey,
    AUTHOR = {Braun, Benjamin},
     TITLE = {Unimodality problems in {E}hrhart theory},
 BOOKTITLE = {Recent trends in combinatorics},
    SERIES = {IMA Vol. Math. Appl.},
    VOLUME = {159},
     PAGES = {687--711},
 PUBLISHER = {Springer, [Cham]},
      YEAR = {2016},
doi={10.1007/978-3-319-24298-9_27}
}

@article{Ferroni,
  title={Examples and counterexamples in {Ehrhart} theory},
  author={Ferroni, Luis and Higashitani, Akihiro},
  journal={Preprint arXiv:2307.10852, to appear in EMS Surv. Math. Sci. },
  year={2024},
doi={DOI 10.4171/EMSS/86},
}

@article{Balletti,
author = {Gabriele Balletti},
title = {A Genetic Algorithm to Search the Space of {Ehrhart} {h*}-Vectors},
journal = {Experimental Mathematics},
volume = {34},
number = {4},
pages = {824--835},
year = {2025},
publisher = {Taylor \& Francis},
doi = {10.1080/10586458.2024.2419526}
}

@article{Mat,
author = {Michał Lasoń and Mateusz Michałek},
title = {Non-Normal Very Ample Polytopes – Constructions and Examples},
journal = {Experimental Mathematics},
volume = {26},
number = {2},
pages = {130--137},
year = {2017},
publisher = {Taylor \& Francis},
doi = {10.1080/10586458.2015.1128370}
}

@Book{beckbook,
 Author = {Beck, Matthias and Robins, Sinai},
 Title = {Computing the continuous discretely. {Integer}-point enumeration in polyhedra.},
 Edition = {2nd},
 FSeries = {Undergraduate Texts in Mathematics},
 Series = {Undergraduate Texts Math.},
 Year = {2015},
 Publisher = {New York, NY: Springer},
 Language = {English},
 DOI = {10.1007/978-1-4939-2969-6},
 zbMATH = {6457081},
 Zbl = {1339.52002}
}

@article{Van,
 author = {Schepers, Jan and Van Langenhoven, Leen},
 title = {Unimodality questions for integrally closed lattice polytopes},
 fjournal = {Annals of Combinatorics},
 journal = {Ann. Comb.},
 issn = {0218-0006},
 volume = {17},
 number = {3},
 pages = {571--589},
 year = {2013},
 language = {English},
 doi = {10.1007/s00026-013-0185-6},
 zbMATH = {6210400},
 Zbl = {1432.52024}
}

@article{IDP-paper,
      title={Examples of {IDP} lattice polytopes with non-log-concave $h^*$-vector}, 
      author={Johannes Hofscheier and Vadym Kurylenko and Benjamin Nill},
      year={2025},
      journal={Preprint arXiv:2505.18896},
      archivePrefix={arXiv},
      primaryClass={math.CO},
      url={https://arxiv.org/abs/2505.18896}, 
}

\end{document}